\documentclass[11pt,reqno,letterpaper]{amsart}
\usepackage{color}
\usepackage[colorlinks=true, allcolors=blue,backref=page]{hyperref}
\usepackage{amsmath, amssymb, amsthm}
\usepackage{mathrsfs}
\usepackage{mathtools}
\usepackage[noabbrev,capitalize,nameinlink]{cleveref}
\usepackage{fullpage}
\usepackage[noadjust]{cite}
\usepackage{graphics}
\usepackage{pifont}
\usepackage{tikz}
\usepackage{bbm}
\usepackage{float} 
\usepackage[T1]{fontenc}

\usetikzlibrary{arrows.meta}

\usepackage{environ}
\usepackage{framed}
\usepackage{url}
\usepackage[linesnumbered,ruled,vlined]{algorithm2e}
\usepackage[noend]{algpseudocode}
\usepackage[labelfont=bf]{caption}
\usepackage{cite}
\usepackage{framed}
\usepackage[framemethod=tikz]{mdframed}
\usepackage{appendix}
\usepackage{graphicx}
\usepackage[textsize=tiny]{todonotes}
\usepackage{tcolorbox}
\usepackage{enumerate}
\allowdisplaybreaks[1]
\usepackage{enumerate}
\usepackage{stmaryrd}
\usepackage[margin=1in]{geometry}

\usepackage[shortlabels]{enumitem}
\crefformat{enumi}{#2#1#3}
\crefrangeformat{enumi}{#3#1#4 to~#5#2#6}
\crefmultiformat{enumi}{#2#1#3}
{ and~#2#1#3}{, #2#1#3}{ and~#2#1#3}

\DeclareSymbolFont{symbolsC}{U}{pxsyc}{m}{n}
\SetSymbolFont{symbolsC}{bold}{U}{pxsyc}{bx}{n}
\DeclareFontSubstitution{U}{pxsyc}{m}{n}
\DeclareMathSymbol{\medcircle}{\mathbin}{symbolsC}{7}

\crefname{equation}{}{} 
\AtBeginEnvironment{appendices}{\crefalias{section}{appendix}} 

\usepackage[color,final]{showkeys} 

\colorlet{refkey}{orange!20}
\colorlet{labelkey}{blue!30}

\numberwithin{equation}{section}
\newtheorem{theorem}{Theorem}[section]
\newtheorem{proposition}[theorem]{Proposition}
\newtheorem{lemma}[theorem]{Lemma}
\newtheorem{claim}[theorem]{Claim}

\newtheorem{corollary}[theorem]{Corollary}

\newtheorem*{question*}{Question}

\theoremstyle{definition}
\newtheorem{definition}[theorem]{Definition}

\newtheorem*{definition*}{Definition}

\theoremstyle{remark}

\newcommand{\mb}{\mathbb}
\newcommand{\mbf}{\mathbf}

\newcommand{\mc}{\mathcal}

\newcommand{\on}{\operatorname}

\let\originalleft\left
\let\originalright\right
\renewcommand{\left}{\mathopen{}\mathclose\bgroup\originalleft}
\renewcommand{\right}{\aftergroup\egroup\originalright}

\allowdisplaybreaks

\newif\ifpublic
\publictrue

\ifpublic

\newcommand{\ignore}[1]{}

\else

\fi

\title{Rainbow Common Graphs Must Be Forests}
\author{Yihang Sun}
\address{Massachusetts Institute of Technology, Cambridge, MA, USA}
\email{kimisun@mit.edu}

\begin{document}

\maketitle
\begin{abstract}
We study the rainbow version of the graph commonness property: a graph $H$ is \emph{$r$-rainbow common} if the number of rainbow copies of $H$ (where all edges have distinct colors) in an $r$-coloring of edges of $K_n$ is maximized asymptotically by independently coloring each edge uniformly at random. $H$ is \emph{$r$-rainbow uncommon} otherwise. We show that if $H$ has a cycle, then it is $r$-rainbow uncommon for every $r$ at least the number of edges of $H$. This generalizes a result of Erd\H{o}s and Hajnal, and proves a conjecture of De Silva, Si, Tait, Tun\c{c}bilek, Yang, and Young. 
\end{abstract}
\section{Introduction}
\subsection{History}
In extremal graph theory, a graph $H$ is \emph{common} if the number of copies of $H$ in any graph $n$ vertex graph $G$ and its complement $\overline{G}$ is minimized asymptotically by the Erd\H{o}s-R\'{e}nyi random graph $\mb{G}(n, 1/2)$. In other words, the minimum number of monochromatic copies of $H$ in any $2$-edge coloring of $K_n$ is asymptotically achieved by coloring each edge independently uniformly at random.
Commonness is extensively studied due to its connection to other homomorphism density inequalities, including the Sidorenko conjecture. In particular, Sidorenko graphs are common \cite{lovasz-book}.

A similar question is asked in anti-Ramsey theory: instead of minimizing the number of monochromatic $H$, we maximize the number of \emph{rainbow} copies of $H$ where all edges have distinct colors \cite{erdos-hajnal,Balogh,desilva}.
\begin{definition}\label{def:common}
For a graph $H$ and $r\in \mb{N}$, we say that $H$ is \emph{$r$-rainbow common} if the maximum number of rainbow copies of $H$ in an $r$-coloring of edges of $K_n$ is achieved asymptotically by coloring each edge independently with a uniform random color. Otherwise, $H$ is \emph{$r$-rainbow uncommon}.

We say that $H$ is \emph{rainbow common} (resp. \emph{rainbow uncommon}) if it is $r$-rainbow common (resp. $r$-rainbow uncommon) for all $r\ge e(H)$, where $e(H)$ denotes the number of edges of $H$.
\end{definition}
All asymptotics are up to a $1+o(1)$ factor as $n\to\infty$. 
This definition of $r$-rainbow commonness is called $r$-anti-common in \cite{desilva} and $r$-rainbow uncommon graphs are called not $r$-anti-common.

If $r<e(H)$, the condition in \cref{def:common} is trivial as there is no rainbow $H$ by the pigeonhole principle. Note that $H$ can be $r$-rainbow common for some $r$ and $s$-rainbow uncommon for $s\ne r$, in which case it is neither rainbow common nor uncommon. However, no such $H$ is known.

It is an old result of Erd\H{o}s and Hajnal \cite{erdos-hajnal} that $K_3$ is $3$-rainbow uncommon: there is a coloring that beats the $2/9$ density of rainbow $K_3$'s obtained by the uniform random $3$-coloring. Erd\H{o}s and S\'{o}s posed the question to determine the maximum density \cite{rodl}. With a flag algebra approach, \cite{Balogh} settled this question and determined that the maximum density is $2/5$ asymptotically.

More recently, \cite{desilva} showed that disjoint unions of stars $K_{1, s}$ are rainbow common and introduced the iterated blowup technique to show that $K_s$ is $\binom{s}{2}$-rainbow uncommon for all $s\in\mb{N}$, thereby generalizing the result of Erd\H{o}s and Hajnal. It is also conjectured that, for every $s\ge 3$, the cycle graph with $s$ edges is $s$-rainbow uncommon and path graph with $s$ edges is $s$-rainbow common \cite{desilva}.

We can study a similar question for solutions to an equation. We seek to $r$-color $\{1, 2, \dots, n\}$ to maximize the number of rainbow solutions $(x_1, \dots, x_d)$ to some equation $\sum_{i=1}^d a_ix_i=0$. The Schur equation $x_1+x_2=x_3$ is studied in \cite{wong} while Sidon equation $x_1+x_2=x_3+x_4$ is studied in \cite{fox-eq,Taranchuk}. 
\subsection{Main results}
In this paper, we focus on the graph theoretic setting.
We extend the result of Erd\H{o}s and Hajnal \cite{erdos-hajnal} and settle the cycle conjecture in \cite{desilva} with a much stronger statement. 
\begin{theorem}\label{thm:main}
If graph $H$ contains a cycle, then it is rainbow uncommon.
\end{theorem}
Since graphs without cycles must be forests, we the following result immediately follows.
\begin{corollary}\label{cor:forest}
If graph $H$ is $r$-rainbow common for any $r\ge e(H)$, then $H$ must be a forest.
\end{corollary}
\subsection{Organization}
In \cref{sec:graphon}, we prove \cref{thm:main} using a graphon perturbation technique inspired by that of \cite{fox,lovasz} to study local Sidorenko properties. 
\subsection{Notation}
Throughout the paper, we consider simple graphs $G$ with vertex set $V(G)$ and edge set $E(G)$. Let $v(G)$ and $e(G)$ denote the size of $V(G)$ and $E(G)$, respectively. Let $[n]:=\{1, 2, \dots, n\}$. Let $(n)_k:=k!\binom{n}{k}$ be the falling factorial. Let binomial coefficient $\binom{n}{k}=0$ if $k>n$.
\subsection*{Acknowledgements}
This research was partially conducted during the Polymath Jr. REU 2022. The author extends profound gratitude towards Gabriel Elvin, for his mentorship, and to professor Adam Sheffer, for directing the program. The author thanks Ashwin Sah, Mehtaab Sawhney, and Yufei Zhao for suggesting the graphon perturbation technique in \cref{sec:graphon}. The author also thanks Milan Haiman, Andrew Huang, and Neha Pant for insights and helpful conversations.
\section{Preliminaries}\label{sec:prelim}
\subsection{Graphons}
In this section, we introduce the notions of graphons and graph homomorphism densities that we need for \cref{thm:main}. We will follow \cite{gtac,lovasz-book} where complete expositions are given.
\begin{definition}\label{def:graphon}  
A \emph{graphon} is a measurable function $W:[0, 1]^2\to [0, 1]$ where $W(x, y)=W(y, x)$.
\begin{itemize}
    \item We define the \emph{associated graphon} $W_G$ of a graph $G$: partition $[0, 1]$ into $v(G)$ equal-length intervals $I_1, \dots , I_{v(G)}$. For $(x, y)\in I_i\times I_j$, define $W_G = \mbf{1}\{(i, j)\in E(G)\}$.
    \item For a graphon $W$, the \emph{$W$-random graph} $\mb{G}(n, W)$ on $[n]$ is defined by sampling independently $x_1, \dots, x_n\in [0, 1]$ and putting an edge $(i, j)$ with probability $W(x_i, x_j)$ independently.
\end{itemize}
For a graphs $G, H$ and graphon $W$, define the \emph{$H$-homomorphism density in $G$ and $W$} as
\[ t(H, G) = \frac{\on{hom}(H, G)}{v(G)^{v(H)}}\quad \text{and}\quad t(H, W) = \int_{[0, 1]^{v(H)}} \prod_{(u, v)\in E(H)} W(x_u, x_v)\prod_{v\in V(H)} dx_v,\]
where $\on{hom}(H, G)$ is the number of graph homomorphisms from $H$ to $G$.
\end{definition}
For example, the Erd\H{o}s-R\'{e}nyi random graph $\mb{G}(n, p)$ is $W$-random graph $\mb{G}(n, W)$ where $W$ is the constant graphon $W(x, y)=p\in [0,1]$. One can similarly view stochastic block models as graphons.

In this language, a graph $H$ is Sidorenko if $t(H, W)\ge t(K_2, W)^{e(H)}$ for every graphon $W$ and $H$ is common if $t(H, W)+t(H, 1-W)\ge 2^{1-e(H)}$ for every $W$ \cite{gtac,lovasz-book}.
Note that $t(H, G)$ is the probability that a uniform random map from $V(H)\to V(G)$ is a graph homomorphism. 
The following simple observation shows that the two definitions of homomorphism densities coincide for graphs.
\begin{lemma}\label{lem:hom-density-eq}
$t(H, W_G)=t(H, G)$ for all graphs $G, H $ and graphon $W_G$ associated to $G$. 
\end{lemma}
The following statement captures the viewpoint that graphons are limits of graphs in the context of homomorphism densities. This uses the notion of left convergence of graphs to graphons.
\begin{proposition}[{\cite[Theorem 4.4.2]{gtac}}]\label{prop:homom-density}
Let $W$ be a graphon. For each $n\in\mb{N}$, let $G_n\sim \mb{G}(n, W)$. Then, $G_n$ left converges to $W$ almost surely, i.e. for all graphs $H$, $t(H, G_n)\to t(H, W)$ as $n\to\infty$.
\end{proposition}
An alternative view point uses the notion of convergence in \emph{cut metric} $\delta_{\square}$ to view graphons as the completion of the space of graphs with respect to $\delta_{\square}$. We refer to \cite{gtac,lovasz-book,szegedy} for more details. 

\subsection{Colorings}
We can view graphs as two colorings of complete graphs. This perspective allows us to naturally extend the notions of graphons and homomorphism densities to edge colorings by considering the graphons associated to the induced subgraph spanned by each color.
\begin{definition}\label{def:coloring-graphon}
For $r\in \mb{N}$, we say that $W=(W_1, \dots, W_r)$ is an \emph{$r$-coloring graphon} if $W_i$ is a graphon for every $i\in [r]$ and, for every $x, y\in [0, 1]$,
\begin{equation}\label{eq:coloring-cond}
     \sum_{i=1}^rW_i(x, y)=1.
\end{equation}
We define the \emph{associated $r$-coloring graphon} of a coloring $\phi:E(K_n)\to [r]$ as $W_\phi = ((W_\phi)_1,\dots, (W_\phi)_r)$ where $(W_\phi)_i$ is the associated graphon of graph $G_i = (V(G), \phi^{-1}(i))$ spanned by edges of color $i$.
\end{definition}
Similar to the Erd\H{o}s-R\'{e}nyi random graph $\mb{G}(n, p)$, the uniform random $r$-coloring can be represented by the $r$-coloring graphon $\mbf{1}/r = (1/r, \dots , 1/r)$. 
We now define an analog of homomorphism densities to count the density of rainbow copies of $H$ in an $r$-coloring graphon $W$ with $r\ge e(H)$.
\begin{definition}\label{def:rainbow-density}
Given a graph $H$ and an $r$-coloring graphon $W=(W_1, \dots, W_r)$, let $\mc{H}$ be the set of injections from $E(H)$ to $[r]$. 
We define the \emph{rainbow homomorphism density of $H$ in $W$} as
\begin{equation}\label{eq:t-coloring}
t(\on{rb}H, W) := \sum_{h\in \mc{H}}\int_{[0, 1]^{v(H)}} \left(\prod_{(u, v)\in E(H)}W_{h(u, v)}(x_u, x_v)\right)\prod_{v\in V(H)}dx_v.
\end{equation}
\end{definition}
Note that $h(u, v)$ is the color of edge $(u, v)$, which occurs with probability $W_h$. 
The condition of $\mc{H}$ that $h$ is injective corresponds exactly to the copy we picked out being rainbow.
The analogs of \cref{lem:hom-density-eq,prop:homom-density} hold for rainbow graph density. For brevity, we only state and prove what we need for \cref{thm:main}. The following lemma is analogous to $W$-random graphs.

\begin{lemma}\label{lem:assoc-coloring}
For any $r, n\in\mb{N}$, given an $r$-coloring graphon $W=(W_1, \dots, W_r)$, we define a random $r$-coloring $\phi\sim \mb{G}(n, W)$ of edges of $K_n$ with vertex set $[n]$ as follows.
\begin{enumerate}
    \item Sample independent $x_1, x_2, \dots , x_n\sim \on{Unif}([0, 1])$. 
    \item For every edge $(u, v)\in E(K_n)$, independently sample $\phi(u, v)\in [r]$ where $\phi(u, v)=i$ with probability $W_i(u, v)$. Color edge $(u, v)$ by color $\phi(u, v)$.
\end{enumerate}
Then, the probability a uniform random copy of $H$ is rainbow under coloring $\phi$ is $t(\on{rb}H, W)$.
\end{lemma}
\begin{proof}
Fix a uniform random copy $H$, we compute the probability that it is rainbow. The rainbow edge colorings of $H$ are precisely $\mc{H}$. For each $h\in\mc{H}$, the probability that $\phi$ colors edge $(u, v)$ of $H$ by color $h(u, v)$ is precisely $W_{h(u, v)}(x_u, x_v)$. By independence of edges $(u, v)$, we take the product and take the expectation over $x_1, \dots, x_d$ independently uniform on $[0, 1]$ to prove \cref{lem:assoc-coloring}.
\end{proof}
\begin{corollary}\label{cor:reduction}
For graph $H$ and $r\ge e(H)$, $H$ is $r$-rainbow uncommon if there exists an $r$-coloring graphon $W=(W_1, \dots, W_r)$ such $t(\on{rb}H, W) > t(\on{rb}H, \mbf{1}/r) = r^{-e(H)}(r)_{e(H)}$.
\end{corollary}
\begin{proof}
By \cref{eq:t-coloring}, we compute that $t(\on{rb}H, \mbf{1}/r) = r^{-e(H)}(r)_{e(H)}$. This is equal to the probability that a uniform random $r$-coloring of edges of $H$ is rainbow. The total number of copies of $H$ in $K_n$ is 
\[ \frac{(n)_{v(H)}}{\left|\on{Aut}(H)\right|} = \Theta\left(n^{v(H)}\right),\]
where $\left|\on{Aut}(H)\right| = \on{hom}(H, H)$ is the number of automorphisms of $H$. By linearity of expectations and \cref{lem:assoc-coloring}, if $t(\on{rb}H, W) > r^{-e(H)}(r)_{e(H)}$, then the expected number of rainbow $H$ under $\phi\sim \mb{G}(n, W)$ is asymptotically greater than that of $\mbf{1}/r$. By the probabilistic method, there exists such a coloring $\phi$ of $E(K_n)$, so $H$ is $r$-rainbow uncommon by \cref{def:common}.
\end{proof}
We can now pass from asymptotics in $n$ to coloring graphons $W$. In \cref{lem:random}, we record this exact number of copies of rainbow $H$ of a uniform random $r$-coloring of $E(K_n)$ on expectation. 
\section{Graphon Perturbation}\label{sec:graphon}
\subsection{Some lemmas}
We prove \cref{thm:main} by constructing a graphon coloring $W$ in \cref{cor:reduction} as a perturbation of the uniform random coloring graphon $\mbf{1}/r$. To do so, we need a few lemmas.
\begin{lemma}\label{lem:f}
Define $f:[0, 1]^2\to \mb{R}$ to be $1$ on $[0, 1/2]^2\cup [1/2, 1]^2$ and $-1$ otherwise.
\begin{enumerate}
	\item For every $x, y\in [0, 1]$, $f(x, y)=f(y, x)$.
	\item For every $t\in [0, 1]$,
\[
\int_0^1 f(x, t)dx = \int_0^1 f(t, x)dx = 0.
\]
	\item For every $s\in \mb{N}$, let $x_{s+1}=x_1$. Then,
\[ \int_{[0, 1]^s} \prod_{i=1}^s f(x_i, x_{i+1})dx_1\dots dx_s =1.\]
\end{enumerate}
\end{lemma}
\begin{proof}
(1) and (2) are clear. For (3), let $Z_i$ be $1$ if $x_i\in [0, 1/2]$ and $-1$ otherwise, so $f(x_i, x_{i+1}) = (-1)^{Z_i+Z_{i+1}}$. Hence, with expectation taken over independent $Z_i\sim \on{Unif}(\{-1, 1\})$, we have that
\[ \int \prod_{i=1}^s f(x_i, x_{i+1})dx_1\dots dx_s =\mb{E} \left[\prod_{i=1}^s (-1)^{Z_i+Z_{i+1}} \right]= \mb{E} \left[(-1)^{2\sum_{i=1}^s Z_i}\right]= 1.\qedhere\]
\end{proof}
Let $f$ be as in \cref{lem:f}. For any graph $G$, define
\begin{equation}\label{eq:IfG}
  \mc{I}_f(G): = \int _{[0, 1]^{v(G)}}\left(\prod_{(x_u, x_v)\in E(G)} f(x_u, x_v) \right) \prod_{v\in V(G)}dx_{v}.
\end{equation}
Then, \cref{lem:f}(3) says that $\mc{I}_f(G)=1$ if $G$ is a cycle graph.
\begin{lemma}\label{lem:IfG-leaf}
For $f$ defined in \cref{lem:f}, $\mc{I}_f(G) = 0$ if $G$ has a leaf.
\end{lemma}
\begin{proof}
Suppose leaf $\ell$ is connected to vertex $k$, then $\int f(x_\ell, x_k)dx_\ell=0$ by \cref{lem:f}(2). Now,
\[ \mc{I}_f(G)=\int {\left(\int f(x_\ell, x_k)dx_\ell\right)} \prod_{(x_u, x_v)\in E(G)\setminus\{(x_k, x_\ell)\}} f(x_u, x_v)  \prod_{v\in V(G)\setminus \{\ell\}} dx_v = 0.\qedhere\]
\end{proof}

\begin{lemma}\label{lem:Ersk}
For every $r\ge s\ge 3$, there exists some $k:=k(r, s)\in [r-1]$ such that
\begin{equation}\label{eq:Ek}
F(r, s, k):=\sum_{i=0}^{s} (-1)^i\binom{k}{i}\binom{r-k}{s-i}k^{s-i}\left(r-k\right)^{i} >0.
\end{equation}
\end{lemma}
\begin{proof}
Let $\ell = r-s\ge 0$. For $s\ge 4$, we have that
\begin{align*}
F(r, s, 2) & = \binom{r-2}{s}2^s-2\binom{r-2}{s-1}2^{s-1}(r-2)+\binom{r-2}{s-2}2^{s-2}(r-2)^2
	\\ & = 2^{s} \binom{r-2}{s-2}\left(\frac{(r-s)(r-s-1)}{s(s-1)}-\frac{(r-s)(r-2)}{s-1}+\frac{1}{4}(r-2)^2\right)
	\\ & = 2^{s} \binom{r-2}{s-2}\left(\frac{\ell(\ell-1)}{s(s-1)}-\frac{\ell (s+\ell-2)}{s-1}+\frac{1}{4}(s+\ell -2)^2\right)
	\\ & = 2^{s} \binom{r-2}{s-2}\left(\frac{s+\ell}{4s}\right)\left((s-4)\ell+(s-2)^2\right)
	\\ & > 0.
	\end{align*}
For $s=3$, we know $r\ge 3$, so
\[ F(r, 3, r-1) = \binom{r-1}{2}(r-1)-\binom{r-1}{3} = \frac{r(r-1)(r-2)}{3} >0.\qedhere\]
\end{proof}

\subsection{Proof of \cref{thm:main}}
Fix any graph $H$ with girth $s\in\mb{N}$ and positive integer $r \ge e(H)$, so
\begin{equation}\label{eq:res3}
r\ge e(H)\ge s\ge 3.
\end{equation}
Hence, we can apply \cref{lem:Ersk} to define $k=k(r, s)\in [r-1]$. Define $\sigma:[r]\to \mb{R}$ by
\begin{equation}\label{eq:sigma}
\sigma (i) = \begin{cases}
	\frac{1}{k} & \textup{if }i\in [k]\\
	-\frac{1}{r-k} & \textup{if }i \in [r]\setminus [k]
\end{cases}.
\end{equation}
Recall $f$ from \cref{lem:f}. For $\varepsilon\in (0, 1/r]$ to be chosen later, define $W_i:[0, 1]^2\to [0, 1]$ by
\begin{equation}\label{eq:Wi}
   W_i(x, y) = \frac{1}{r}+\varepsilon \sigma(i)f(x, y)
\end{equation}
for each $i\in [r]$. Clearly, $W_i$ is measurable for each $i$. For every $x, y\in [0, 1]$, $W_i(x, y)=W_i(y, x)$ as $f(x, y)=f(y, x)$. Since $\sigma_i\in [-1, 1]$ and $f(x, y)\in [-1, 1]$, choosing any $\varepsilon\in [0, 1/r]$ satisfies $W_i(x, y)\in [0, 1]$. Hence, $W_i$ is a graphon for each $i$. Now, for any $x, y\in [0, 1]$, we compute that
\[
\sum_{i=1}^r W_i(x, y) = 1+\varepsilon f(x, y)\sum_{i=1}^r \sigma(i)=1,
\]
so $W=(W_1, \dots, W_r)$ is an $r$-coloring graphon. Recall that $\mc{H}$ is the set of injections from $E(H)$ to $[r]$. By \cref{cor:reduction}, it suffices to show that $t(\on{rb}H, W) >t(\on{rb}H, \mbf{1}/r)= r^{-e(H)}|\mc{H}|$. Noye that
\begin{align*}
t(\on{rb}H, W) & = \sum_{h\in \mc{H}}\int \prod_{(u, v)\in E(H)}\left[\frac{1}{r}+\varepsilon \sigma(h(x_u, x_v))f(x_u, x_v)\right]\prod_{v\in V(H)}dx_v.
\\ & = \sum_{h\in \mc{H}}\int \sum_{G\subset H} r^{e(G)-e(H)}\varepsilon^{e(G)}\prod_{(u, v)\in E(G)} \sigma(h(u,v ))f(x_u, x_v)\prod_{v\in V(H)}dx_v
\\ & = \sum_{G\subset H}  r^{e(G)-e(H)} \varepsilon^{e(G)}\sum_{h\in \mc{H}} \int  \prod_{(u, v)\in E(G)} \sigma(h(u, v))f(x_u, x_v)\prod_{v\in V(H)}dx_v,
\end{align*}
where in the second step we expand the first product and index the terms where we pick out the second term in the square brackets by subgraph $G\subset H$. The summand corresponding to $G=\varnothing$ is exactly $ r^{-e(H)}|\mc{H}|=r^{-e(H)}(r)_{e(H)}=t(\on{rb}H, \mbf{1}/r)$. Since $f$ does not depend on $h$ and $\sigma$ does not depend on $x_v$, we can factor out the terms to rewrite
\smallskip
\begin{align*}
& t(\on{rb}H, W) - t(\on{rb}H, \mbf{1}/r) 
\\ & = \sum_{\varnothing \neq G\subset H}  r^{e(G)-e(H)} \varepsilon^{e(G)}   \sum_{h\in \mc{H}} \prod_{(u, v)\in E(G)} \sigma(h(u, v))\left(\int \prod_{(u, v)\in E(G)}  f(x_u, x_v)\prod_{v\in V(G)}dx_v\right)
\\ & = \sum_{\varnothing \neq G\subset H}  r^{e(G)-e(H)} \varepsilon^{e(G)} \left(\int \prod_{(u, v)\in E(G)}  f(x_u, x_v)\prod_{v\in V(G)}dx_v\right) \left(\sum_{h\in \mc{H}} \prod_{(u, v)\in E(G)} \sigma(h(u, v))\right)
\\ & = \sum_{\varnothing \neq G\subset H}  r^{e(G)-e(H)} \varepsilon^{e(G)} \mc{I}_f(G)\left(\sum_{h\in \mc{H}} \prod_{(u, v)\in E(G)} \sigma(h(u, v))\right)
\end{align*}
by \cref{eq:IfG}. Now, by \cref{lem:IfG-leaf}, $\mc{I}_f(G)$ is zero unless $G\subset H$ has no leaf. Since $H$ has girth $s$, the only subgraph $G\subset H$ such that $\mc{I}_f(G)\ne 0$ and $e(G)\le s$ are those isomorphic to the cycle graph on $s$ vertices, written $G\simeq C_s$. There, $\mc{I}_f(G)=1$ by \cref{lem:f}(3).  We make the following claim.
\begin{claim}\label{claim:key}
$Q(G):=\sum_{h\in \mc{H}} \prod_{(u, v)\in E(G)} \sigma(h(u, v)) >0$ for all $G\subset H$ with $G\simeq C_s$.
\end{claim}
We first show how \cref{claim:key} finishes the proof.
For all other $G\not\simeq C_s$ such that $\mc{I}_f(G)\ne 0$, $e(G) > s$. Once we fixed $r$ and $H$, everything is fixed except for $\varepsilon\in (0, 1/r]$. Then, by taking $\varepsilon$ sufficiently small, the terms with $G\simeq C_s$ dominates and is positive, i.e.
\begin{align*}
 t(\on{rb}H, W) - t(\on{rb}H, \mbf{1}/r) & = \sum_{\varnothing \neq G\subset H}  r^{e(G)-e(H)} \varepsilon^{e(G)} \mc{I}_f(G)\left(\sum_{h\in \mc{H}} \prod_{(u, v)\in E(G)} \sigma(h(u, v))\right)
 \\ & = \varepsilon^s\left(\sum_{G\subset H:G\simeq C_s} r^{s-e(H)}  Q(G)\right) + O_{\varepsilon \to 0}(\varepsilon^{s+1})
 \\ & \geq \frac{1}{2}\varepsilon^s\sum_{G\subset H:G\simeq C_s} r^{s-e(H)}  Q(G)
 \\ & > 0,
\end{align*}
where the second to last step holds by choosing sufficiently small $\varepsilon$ with respect to $r$ and $H$, and the last step holds by \cref{claim:key}. By \cref{cor:reduction}, we are done once we prove \cref{claim:key}.

\begin{proof}[Proof of \cref{claim:key}.]
Recall that $h(u, v)$ specifies the color of edge $(u, v)$. By \cref{eq:sigma}, $\sigma(h(u, v))$ depends only on whether $h(u, v)\in [k]$ or $h(u, v)\in [r]\setminus [k]$.
Suppose $h$ assigns $i$ edges of $G\simeq C_s$ to have colors in $[k]$, then it assigns colors among $[r]\setminus [k]$ to the other $s-i$ edges of $G$. For each such $h$
\[ \prod_{(u, v)\in E(G)} \sigma(h(u, v)) = \left(\frac{1}{k}\right)^i\left(-\frac{1}{r-k}\right)^{s-i}.\]

We count the number of such $h\in\mc{H}$. Since $h$ is injective, there are $\binom{k}{i}$ ways to pick out $i$ colors in $[k]$ and $\binom{r-k}{s-i}$ ways to pick out $s-i$ colors in $[r]\setminus [k]$. Then, there are $s!$ ways to assign these $s$ total colors to edges of $G$. Now, it remains to assign colors to $E(H)\setminus E(G)$. There are $(r-s)_{e(H)-e(G)}$ ways. Hence, in total, the number of such $h\in \mc{H}$ is
\[ \binom{k}{i}\binom{r-k}{s-i}\cdot s!\cdot (r-s)_{e(H)-s}.\]

Therefore, for $G\simeq C_s$, we have that
\begin{align*}
Q(G) & = \sum_{i=1}^s \left(\frac{1}{k}\right)^i\left(-\frac{1}{r-k}\right)^{s-i} \binom{r-k}{s-i}\cdot s!\cdot (r-s)_{e(H)-s}  
\\ & = k^s(r-k)^s \cdot s!\cdot (r-s)_{e(H)-s}  F(r, s, k),
\end{align*}
where we recall $F$ from \cref{eq:Ek}. By \cref{eq:res3}, the conditions of \cref{lem:Ersk} hold, so we chose $k\in [r-1]$ such that $F(r, s, k)>0$. Clearly, $k^s$, $(r-k)^s$ and $s!$ are all strictly positive. We show that falling factorial $(r-s)_{e(H)-s} >0$. This is clear combinatorially since it is the number of ways to extend $h$ from $E(G)$ to $E(H)$. More carefully, it is obvious if $r>s$. If $r=s$, then as $r\ge e(H)\ge s$ by \cref{eq:res3}, so $(r-s)_{e(H)-s}= (0)_0 = 1$. Therefore, $Q(G)>0$ for $G\simeq C_s$.
\end{proof}
This concludes the proof of \cref{thm:main}.
\bibliographystyle{amsplain0.bst}
\bibliography{main.bib}

\end{document}